\documentclass[12pt,draftcls,onecolumn]{IEEEtran}

\IEEEoverridecommandlockouts
\usepackage{pgf,tikz}
\usetikzlibrary{patterns}  
\usepackage{times} 
\usepackage{color}
\usepackage{graphicx}
\usepackage{setspace}
\usepackage{bbm}
\usepackage{mathdots,mathrsfs}
\usepackage{amssymb,latexsym,amsfonts,amsmath,cite}
\usepackage{stmaryrd}
\usepackage{caption}
\usepackage[dvips]{psfrag}
\usepackage{setspace}
\usepackage{amsmath}
\usepackage{url}
\usepackage[linesnumbered,ruled]{algorithm2e} 

\newtheorem{theorem}{Theorem}
\newtheorem{proof}{Proof}

\newtheorem{corollary}{Corollary}

\newtheorem{remark}{Remark}

\newtheorem{example}{Example}
\newtheorem{assumption}{Assumption}
\newtheorem{definition}{Definition}
\newcommand{\R}{{\mathbb{R}}}

\newcommand{\T}{{\text{T}}}

\usepackage{pgf,tikz}
\DeclareMathOperator{\Real}{Re}

\DeclareMathOperator{\Tr}{\mathrm{Tr}}
\DeclareMathOperator{\diag}{\mathrm{diag}}

\newcommand{\G}{{\mathcal{G}}}
\newcommand{\V}{{\mathcal{V}}}
\newcommand{\EE}{{\mathcal{E}}}
\DeclareMathOperator{\CT}{\text{H}}

\begin{document}
\title{\LARGE \bf Stability and Robustness Analysis of\\ Commensurate Fractional-order Networks}

\author{ Milad Siami
\thanks{M. Siami is with Electrical \& Computer Engineering Department, Northeastern University, Boston, MA 02115 USA (e-mail:  {\tt\small m.siami@northeastern.edu}).}
}

\maketitle

\begin{abstract}
Motivated by biochemical reaction networks, a generalization of the classical secant condition for the stability analysis of cyclic interconnected commensurate fractional-order systems is provided. The main result presents a sufficient condition for stability of networks of cyclic interconnection of fractional-order systems when the digraph describing the network conforms to a single circuit. {The condition becomes necessary under a special situation where coupling weights are uniform.} We then investigate the robustness of fractional-order linear networks. Robustness performance of a fractional-order linear network is quantified using the $\mathcal H_2$-norm of the dynamical system. Finally, the theoretical results are confirmed via some numerical illustrations.\\
\end{abstract}


\section{Introduction}

Due to growing interest in design, analysis, and control of complex dynamical networks, it is important to be able to describe and model these systems accurately \cite{magin2019fractional}. Derivation from first principles and Gaussian statistical approaches leads to many of the canonical models and methods ubiquitous in physics, engineering, and even the biological sciences. However, there remain many complex systems that have eluded quantitative analytic descriptions (e.g., non-Brownian diffusion in complex dynamical networks; description of anomalous, long-range, and non-local properties). For example, soft/flexible robots are entailed by the elastic property of porous media and long-term memory; these properties cannot be fully captured by ordinary differential equations \cite{rahmani2016, rus2015}. Moreover, viscoelastic properties are typical for a wide variety of biological networks \cite{efremov2020measuring,lieleg2009cytoskeletal}. Fractional-order systems, widely regarded as an extension of integer-order systems, are those dynamical systems whose state space representations involve non-integer derivatives of the states. In \cite{SiamiFrac2013,SiamiFrac2008}, we investigated important problems regarding fractional-order systems, which are their stability analysis, oscillation analysis, and chaos control. One of the significant potentials of fractional calculus has been considered for the description of anomalous, non-Brownian diffusion in complex dynamical networks. 

For different applications of fractional-order systems, comprehensive understanding of the properties of such systems is of great importance. Despite this importance, until now only a few studies have been done in order to investigate the properties of fractional-order models (cf. \cite{rahmani2016,magin2019fractional,SiamiFrac2008}). Moreover, the class of cyclic networks has been investigated heavily in the context of systems biology \cite{Betz65, ChandraBD11, arcaksontag06, Selkov75}.  Cyclic feedback structures have classically been used to model autoregulatory feedback loops in gene networks, metabolic pathways, and cell signaling \cite{ChandraBD11,tyson1978dynamics}.  
The classical secant condition applies to first-order linear time-invariant (FLTI) systems that are obtained as negative feedback cycles. 
This paper presents a generalization of the secant criterion that explicitly displays the range of parameters that guarantee the stability of cyclic interconnected scalar fractional-order systems. This approach is now throwing new light onto classical results in mathematical biology and suggesting new directions for further research in that field.

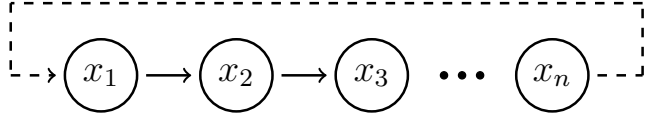
\begin{figure*}[t]
\centering
\scalebox{1.2}{
\begin{tikzpicture}
\draw [ thick] (-2.5,0) circle [radius=0.4];
\draw [ thick] (-1,0) circle [radius=0.4];
\draw [ thick] (.5,0) circle [radius=0.4];
\draw [  thick] (2.5,0) circle [radius=0.4];
\draw [->, thick] (-2,0) -- (-1.5,0);
\draw [->,  thick] (-.5,0) -- (0,0);
\draw [fill] (1.3,0) circle [radius=.04];
\draw [fill] (1.5,0) circle [radius=.04];
\draw [fill] (1.7,0) circle [radius=.04];
\draw [ thick, dashed] (3,0) -- (3.5,0);
\draw [ thick, dashed] (3.5,0) -- (3.5,.8);
\draw [ thick, dashed] (3.5,.8) -- (-3.5,.8);
\draw [ thick, dashed] (-3.5,.8) -- (-3.5,0);
\draw [ ->,thick, dashed] (-3.5,0) -- (-3,0);
\node[] at (-2.5,0) {$x_1$};
\node[] at (-1,0) {$x_2$};
\node[] at (.5,0) {$x_3$};
\node[] at (2.5,0) {$x_n$};
\end{tikzpicture}}
  	\caption{{\small Schematic diagram of negative feedback noisy cyclic system described by \eqref{eq66}. The dashed link indicates a negative (inhibitory) feedback signal.}}
  	\label{fig_22}
\end{figure*}

In this paper, we first develop some basic principles to study {commensurate} fractional-order systems characterized by a cyclic interconnection structure as depicted in Figure \ref{fig_22}. We focus on the Caputo definition of fractional-order systems \cite{SiamiFrac2008, SiamiFrac2010, SiamiFrac2013}. We develop basic principles to study the stability of fractional-order systems characterized by spatially invariant and cyclic interconnection structures. We then gain insights from the combination of network science and fractional-order calculus for analyzing the robustness of networks of interconnected commensurate FLTI subsystems.


A framework for robustness analysis of commensurate FLTI networks under the influence of external anomalous disturbances is developed in Section \ref{Sec:5}. The squared $\mathcal H_2$-norm of the output of an FLTI stable network is adopted as a performance measure to quantify performance deterioration of the network. This performance measure
is equal to the square of the $\mathcal H_2$-norm of the network from the disturbance input to the output \cite{SiamiTAC2016, SiamiAutomatica}. {
The (squared) $\mathcal H_2$ norm of these FLTI systems can be interpreted as the impulse response energy.
One of our contributions in this paper is to show how
the performance measure relates to the underlying graph of the commensurate FLTI network and we obtain a closed-form solution based on eigenvalues of the state matrix.

One of the central problems in control theory is determining to what extent uncertain exogenous inputs can steer the trajectories of a dynamical network away from its working equilibrium point. The primary challenge is to introduce meaningful and viable performance and robustness measures that can capture the network's essential characteristics. Our proposed spectral-based measure reveals the critical role and contribution of fast and slow dynamic modes of a CFLTI system in the best and worst achievable robustness performance under white noise excitation. Furthermore, in \cite{SiamiTAC18,siami2017growing,siami2017abstraction}, we investigate several useful functional properties of performance measures (e.g., convexity, monotonicity, homogeneity) that allow us to utilize them in network synthesis problems.}

\section{Notations, Definitions, and Basic Concepts}

\allowdisplaybreaks

{\subsection{Spectral Graph Theory}

We first review the basic concepts. An undirected graph herein is defined by a triple $\G = (\V, \EE,w)$, where $\V$ is the set of nodes, $\EE \subseteq \big\{\{i,j\}~\big|~ i,j \in \V, ~i \neq j \big\}$ is the set of links, and $w: \EE \rightarrow  \R_{++}$ is the weight function. 
{The adjacency matrix $A = [a_{ij}]$ of graph $\G$ is defined in such a way that $a_{ij} = w(e)$ if $e=\{i,j\} \in \EE$, and $a_{ij}=0$ otherwise. The Laplacian matrix of $\G$ is defined by $L := \Delta - A$, where $\Delta=\diag[d_{1},\ldots,d_{n}]$ and $d_i$ is degree of node $i$.}  We denote the set of Laplacian matrices of all connected weighted graphs with $n$ nodes by $L$. Since $\G$ is both undirected and connected, the Laplacian matrix $L$ has $n-1$ strictly positive eigenvalues and one zero eigenvalue. We Assume that $0 = \lambda_1 < \lambda_2 \leq \ldots \leq \lambda_n$ are eigenvalues of Laplacian matrix $L$. Eigenvalues of $n$-by-$n$ matrix $A$ are shown by $\lambda_i(A)$ for $i=1,2,\cdots,n$.
}

\subsection{Fractional Calculus}
In this paper, we focus on the Caputo definition. This definition is given by
\begin{eqnarray}
&&\frac{d^{\alpha}}{d t^{\alpha}}f(t)=\frac{1}{\Gamma(n-\alpha)}\int_{0}^{t}(t-\tau)^{n-\alpha-1}\frac{d^n}{d t^n}f(\tau)d \tau,
\end{eqnarray}
where $\alpha$ is a positive real number, $\Gamma(.)$ is the Gamma function and $n$ is the first integer not less than $\alpha$ ({\it i.e.}, $\lceil \alpha \rceil = n$).
An FLTI system can be represented by the following pseudo state space form
\begin{equation}
\Sigma:
    \begin{cases}
	\frac{d^{\bar \alpha}}{d t^{\bar \alpha}}x(t) ~=~ A x(t) ~+~ B {\xi(t)}, \\
	y(t) ~=~ C x(t),
	\label{eq2}
	\end{cases}
\end{equation}
where $x \in \mathbb {R}^n$, $u \in \mathbb{R}^m$,  $y \in \mathbb{R}^p$, $A \in \mathbb{R}^{n \times n}$, $B \in \mathbb {R}^{n \times m}$, $C \in \mathbb{R}^{p \times n}$, and $ \frac {d^{\bar{\alpha}}}{d t^{\bar{\alpha}}}$ refers to the Caputo derivative where $\bar \alpha =[ \alpha_1 \cdots \alpha_n]^\top$  indicates fractional order settled in the range {$(0, 2)^n$}. The size of vector $x$,  $n$, is called the inner dimension of system (\ref{eq2}). If $\alpha = \alpha_1 = \cdots = \alpha_n$, system (\ref{eq2}) is called a commensurate order system. Also, if $\alpha_i$'s are rational numbers, system (\ref{eq2}) is called a rational order system. The stability of commensurate order systems is investigated in the following  theorem from \cite{matignon1996stability}.
\begin{theorem}
\label{th-1}
Consider the commensurate order system given by
\begin{eqnarray}
	\frac{d^{\alpha}}{d t^{\alpha}}x(t) ~=~ A x(t),
	\label{eq3}
\end{eqnarray}
where $x \in \mathbb{R}^ n$,  $A \in \mathbb{R}^{n \times n}$ and $\bar \alpha =[ \alpha \cdots \alpha]$. System (\ref{eq3}) is asymptotically stable if and only if $|\text{arg}(\lambda)|>\frac{\alpha \pi}{2}$ is satisfied for all eigenvalues $\lambda$ of matrix $A$. Also, this system is stable if and only if $|\text{arg}(\lambda)|\geq \frac{\alpha \pi}{2}$ is satisfied for all eigenvalues $\lambda$ of matrix $A$ and those critical eigenvalues satisfying conditions $|\text{arg}(\lambda)| = \frac{\alpha \pi}{2}$ have equal geometric and algebraic multiplicities. 
\end{theorem}
%

\begin{definition}
\label{def:marg}
A stable commensurate order system defined by (\ref{eq2}) is called marginally stable if and only if
matrix $A$ has at least one eigenvalue satisfying condition $| \text{arg}(\lambda)| = \frac {\alpha \pi}{2}$. 
\end{definition}

\section{Stability Analysis of Cyclic Interconnected Fractional-Order Systems}
\label{sec:III}
In this section, we introduce a stability problem for networks of cyclic interconnection of fractional-order systems.
The class of cyclic feedback systems typically arises in a sequence of biochemical reactions wherein the end product is necessary to power and catalyze its own production while intermediate products activate subsequent reactions \cite{SiamiBio2020}.

Consider an FLTI system $\mathcal G_i$ represented by a state-space model of the form
\begin{eqnarray}
	\frac{d^{\alpha}}{d t^{\alpha}}x_i(t) = -a_i x_i(t) + u_i(t),\quad y_i(t) = c_ix_i(t),
	\label{eq5}
\end{eqnarray}
where $u_i(t)$, $y_i(t)$, and $x_i(t)$ denote its input, output, and state, respectively. 
Now, consider the following cyclic system as depicted in Figure \ref{fig_22}
\begin{eqnarray}
	\frac{d^{\alpha} x_1}{d t^{\alpha}} &=& -a_1 x_1 - y_n, \nonumber \\ 
	\frac{d^{\alpha} x_2}{d t^{\alpha}} &=& -a_2 x_2 + y_1,\nonumber \\
												& \vdots& \nonumber \\
	\frac{d^{\alpha} x_n}{d t^{\alpha}} &=& -a_n x_2 + y_{n-1},
	\label{eq66}
\end{eqnarray}
where $a_i$ and $c_i$ are strictly positive numbers.

Using (\ref{eq5}) and (\ref{eq66}), it follows that

\begin{eqnarray}
\frac{d^{\alpha} x}{d t^{\alpha}} &=& A x,
\label{network}
\end{eqnarray}
where $x \in \mathbb{R}^ n$, $\alpha \in (0,2)$ and
\begin{eqnarray}
	A = \left[
	\begin{array}{cccccc}
	 -a_1&0&\cdots&0&-c_n\\
	  c_1&-a_2&\cdots&0&0\\
	 \vdots &&\ddots &&\vdots&\\
 	  0&0&\cdots&-a_{n-1}&0\\
  	 0&0&\cdots&c_{n-1}&-a_n
	\end{array}
	\right].
	\label{eq6}
\end{eqnarray}
\subsection{Stability Results}
\label{main}
The following theorem presents a necessary and sufficient condition for stability of networks of cyclic interconnection of fractional-order systems when the digraph describing the network conforms to a single circuit. { We then use this result later in Subsection \ref{sec:IV} to investigate the robustness of these networks. }
\begin{theorem}
\label{th-main}
Considering the commensurate fractional-order cyclic network (\ref{network}) {with $\alpha \in (0,2)$}, let us define 
\begin{equation}
\mathfrak a := \sqrt[n]{a_1a_2\cdots a_n},
\end{equation}
and 
\begin{equation}
\mathfrak c := \sqrt[n]{c_1c_2\cdots c_n}.
\end{equation}
The cyclic network \eqref{network} is asymptotically stable if $\alpha \leq\frac{2}{n}$. Moreover, for $\alpha > \frac{2}{n}$ the system is stable if
\begin{eqnarray}
 \gamma :=	\frac{\mathfrak c}{\mathfrak a} < \frac{\sin(\frac{\alpha \pi}{2})}{\sin({\frac{\alpha \pi}{2}-\frac{\pi}{n}})}.
	\label{eq4}
\end{eqnarray}
Further, when the $a_i$'s are identical, condition (\ref{eq4}) is also necessary for stability.
\end{theorem}
\begin{proof}
First, assume that 
\[a_1~=~a_2=~\cdots~=~a_n~=~a.\]
The characteristic equation of $A$ is obtained as follows
\[(\lambda+a)^n+c_1c_2 \cdots c_n~=~0.\]
Hence, $A$ has the following eigenvalues 
\begin{eqnarray}
\lambda_k&=&-a +  {\mathfrak c}\,{\rm e}^{j(\frac{\pi}{n}+\frac{2\pi k}{n})},
\end{eqnarray}
where $ k=0,1,\cdots,n-1$ and $\mathfrak c =\sqrt[n]{c_1c_2 \cdots c_n}$.
As can be seen in Figure \ref{fig:block}, system (\ref{network}) is asymptotically stable if and only if $|\text{arg}(\lambda)| > \frac{\alpha \pi}{2}$ is satisfied for all eigenvalues $\lambda$ of matrix $A$ (cf. Theorem \ref{th-1}). Now by writing the law of sines for the yellow triangle in Figure \ref{fig:block}, we get that {if $\frac{\alpha \pi}{2}> \frac{\pi}{n}$}

\begin{eqnarray}
	&&\frac{\mathfrak c}{\sin({\pi-\frac{\alpha \pi}{2}})} < \frac{a}{\sin({\frac{\alpha \pi}{2}-\frac{\pi}{n}})}.
	\label{eq444}
\end{eqnarray}

\begin{figure}[h]
\centering
\scalebox{1}{
\begin{tikzpicture}
\fill[step=.10cm, pattern= north west lines] (0,0) -- (3,3) --(4,3)--(4,-3)--(3,-3); 
\path [fill=yellow] (-2,0) -- (0,0) --(2,2);
\draw [<->] (-4,0) -- (4,0);
\draw [<->] (0,3) -- (0,-3);
\draw [thick] (0,0) -- (3,3);
\draw [thick] (0,0) -- (3,-3);
\draw [thick, dashed] (-2,0) -- (1,1.5) -- (2,2);
\draw [thick, red] (1,1.5)--(2,2);
\draw [fill] (1,1.5) circle [radius=.1];
\draw [thick] (1,0) arc [radius=1, start angle=0, end angle= 45];
\draw [thick] (-1,0) arc [radius=1, start angle=0, end angle= 25];
\node [] at (3,1) {Unstable};
\node [] at (3,-1) {Unstable};
\node [above right] at (1,.1) {$\frac{\alpha \pi}{2}$};
\node [above right] at (-1,0) {$\frac{\pi}{n}$};
\node [below left] at (-2,0) {$-a$};
\node [] at (-.5,1) {$r$};
\draw [fill] (1,-1.5) circle [radius=.1];
\draw [thick, dashed] (-2,0) -- (1,-1.5);
\draw [thick, dashed] (-2,0) -- (-1.5,2.9);
\draw [fill] (-1.5,2.9) circle [radius=.1];
\draw [thick] (-1.3,.4) arc [radius=.7, start angle=28, end angle= 85];
\node [above right] at (-1.7,.6) {$\frac{2\pi}{n}$};
\draw [fill] (-1.5,-2.9) circle [radius=.1];
\draw [thick, dashed] (-2,0) -- (-1.5,-2.9);
\node[] at (-2.5,1) {$\cdots$};
\node[] at (-2.5,-1) {$\cdots$};
\end{tikzpicture}}
\caption{Stable and unstable regions for system \eqref{network} when the $a_i$'s are identical. Black circles show the poles of the system in the complex plane. The dashed area shows the unstable region in the complex plane.}
\label{fig:block} 
\end{figure}
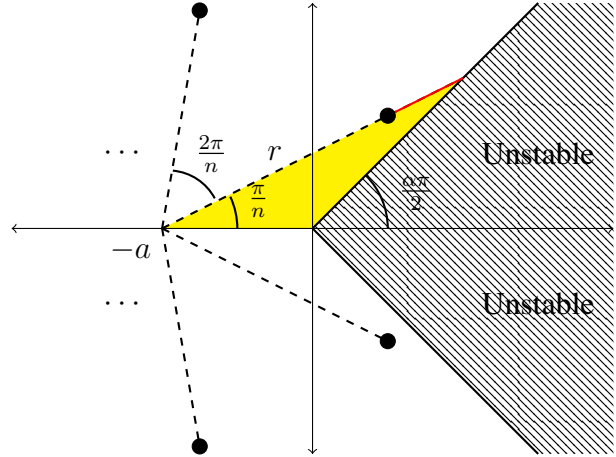

Hence, using (\ref{eq444}), we obtain (\ref{eq4}).   

In the other case, it is easy to see that, if (\ref{network}) is asymptotically stable, then
\begin{eqnarray}
P(\lambda)~=~\prod_{i=1}^{n}(\lambda+a_i)+c_1c_2 \cdots c_n~\ne~0,
\end{eqnarray}
for $|\text{arg}(\lambda)| \leq \frac{\alpha \pi}{2}$.
The asymptotic stability criterion reduces to
\begin{eqnarray}
&&0 ~\ne~ \prod_{i=1}^{n}(\omega {\rm e}^{j\frac{\alpha \pi}{2}}+a_i)+c_1c_2 \cdots c_n, 
\label{eq:200}
\end{eqnarray}
where $\omega \geq 0$.
Using \eqref{eq:200} and the law of sines in Figure \ref{fig:block2}, we get
\begin{eqnarray}
	&&\sum_{i=1}^{n}\theta_i~=~n\frac{\alpha \pi}{2}-\pi,\\
	&&\prod_{i=1}^{n}a_i\frac{\sin(\frac{\alpha \pi}{2})}{\sin \theta_i}~=~(a_1\cdots a_n)\prod_{i=1}^{n}\frac{\sin(\frac{\alpha \pi}{2})}{\sin \theta_i}> \mathfrak c.
	\label{eq44}
\end{eqnarray}
Clearly all $\theta_i$'s enter symmetrically in (\ref{eq44}) and therefore the minimum occurs when $\theta_1=\cdots=\theta_n$ (see Figure \ref{fig:block2}). {Therefore, it is asymptotically stable as prescribed in (\ref{eq4}).}

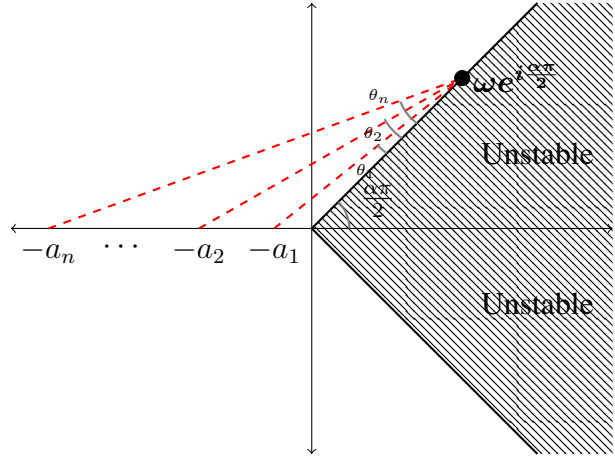
\begin{figure}[h]
\centering
\scalebox{1}{
\begin{tikzpicture}
\fill[step=.10cm, pattern= north west lines] (0,0) -- (3,3) --(4,3)--(4,-3)--(3,-3); 
\draw [<->] (-4,0) -- (4,0);
\draw [<->] (0,3) -- (0,-3);
\draw [thick] (0,0) -- (3,3);
\draw [thick] (0,0) -- (3,-3);
\draw [thick, dashed, red] (-1.5,0) -- (2,2);
\draw [thick, dashed, red] (-3.5,0) -- (2,2);
\draw [thick, dashed, red] (-.5,0) -- (2,2);
\draw [thick, gray] (.5,0) arc [radius=.5, start angle=0, end angle= 45];
\draw [thick,gray] (1,1) arc [radius=1, start angle=235, end angle= 225];
\draw [thick,gray] (1.2,1.2) arc [radius=.8, start angle=235, end angle= 210];
\draw [thick,gray] (1.4,1.4) arc [radius=.6, start angle=235, end angle= 200];
\node [] at (3,1) {Unstable};
\node [] at (3,-1) {Unstable};
\node [above right] at (.5,0) {$\frac{\alpha \pi}{2}$};
\node [right] at (2,2) {$\boldsymbol{\omega e^{i\frac{\alpha \pi}{2}}}$};
\node [below left] at (1,1) {\tiny{$\theta_1$}};
\node [below left] at (1.1,1.5) {\tiny $\theta_2$};
\node [below left] at (1.2,2) {\tiny $\theta_n$};
\node [below] at (-.5,0) {$-a_1$};
\node [below] at (-1.5,0) {$-a_2$};
\node [below] at (-2.5,0) {$\cdots$};
\node [below] at (-3.5,0) {$-a_n$};
\draw [fill] (2,2) circle [radius=.1];
\end{tikzpicture}}
\caption{Stable and unstable regions for system (\ref{network}). Angles $\theta_1, \cdots, \theta_n$ are depicted in this plot. }
\label{fig:block2} 
\end{figure}

\end{proof}

\begin{remark}
Assume that $\alpha=1$. Then the proposed stability test \eqref{eq4} recovers the secant criterion that was derived for cyclic LTI systems \cite{sontag2006}.
\end{remark}

\begin{remark}
 When $a_1=\cdots=a_n$, the system (\ref{network}) is marginally stable if and only if 
 \begin{eqnarray}
	\gamma = \frac{\mathfrak c}{\mathfrak a} =\sqrt[n]{\frac{c_1 \cdots c_n}{a_1 \cdots a_n}} = \frac{\sin(\frac{\alpha \pi}{2})}{\sin({\frac{\alpha \pi}{2}-\frac{\pi}{n}})},
\end{eqnarray}
with exactly one pair of eigenvalues on the boundary line $|\text{arg}(\lambda)| = \frac{\alpha \pi}{2}$.
 \end{remark}

\begin{remark}[Role of number of interconnections]
Figure \ref{fig:my_label0} depicts the range of parameter $\gamma$ that guarantees stability for any given {$\alpha \in (0,1]$} and $n \in \{5,10, 20\}$.
In Figure \ref{fig:my_label0}, the curves show the upper bound on parameter $\gamma$ ({\it i.e.}, the right-hand side of \eqref{eq4}) for any given {$\alpha \in (0,1]$} and $n \in \{5,10, 20\}$. {This plot also {shows} the dependence of network stability on network size. As $n$ increases, the region for the parameter $\gamma$ that guarantees stability shrinks. For example, for $n=5$ this region is the union of the left-hand side of dashed asymptote $y=0.4$ and the area under the blue curve (light blue area). However, for $n=10$, this region reduces to the union of the left-hand side of dashed asymptote $y=0.2$ and the area under the orange curve (blue area). We should note that for $\alpha \leq \frac{2}{n}$ (see the red dashed asymptotes for $n=5,10$ and $20$) the system \eqref{network} is stable, independent of values of $\gamma$.}
\end{remark}

\begin{figure*}
    \centering
\scalebox{0.4}{
    \includegraphics[trim={0cm 4cm 0 4cm},clip]{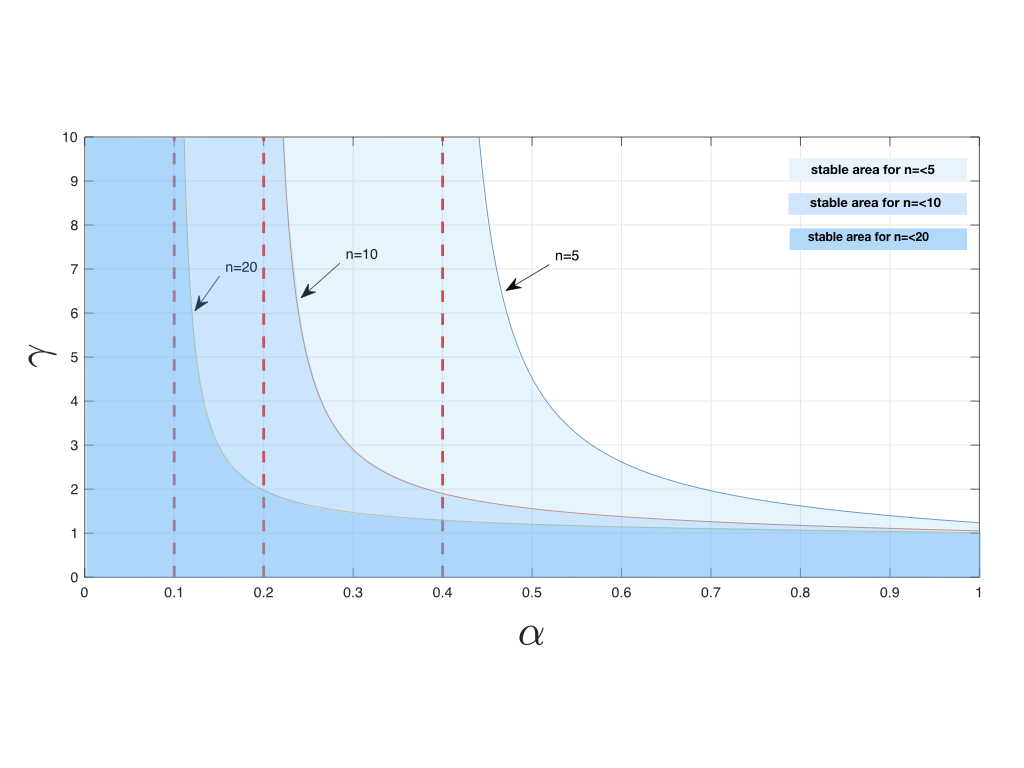}}
    \caption{This plot depicts the upper bound on parameter $\gamma$ (i.e., the right-hand side of \eqref{eq4}) versus parameter $\alpha$ for $n \in \{5,10,20\}$. {This plot also {illustrates} the dependence of network stability on network size. We should note that for $\alpha \leq \frac{2}{n}$ (see the red dashed asymptotes for $n=5,10$ and $20$) the system \eqref{network} is stable, independent of values of $\gamma$. As $n$ increases, the region for the parameter $\gamma$ that guarantees stability shrinks. For example, $n=5$ this region is the union of the left-hand side of dashed asymptote $y=0.4$ and the area under the blue curve (light blue area). However, for $n=10$, this region reduces to the union of the left-hand side of dashed asymptote $y=0.2$ and the area under the orange curve (blue area).} }
    \label{fig:my_label0}
\end{figure*}

\begin{remark}
Unlike a small-gain condition which would restrict the right-hand side of inequality \eqref{eq4} to be 1, the generalized secant criterion \eqref{eq4} also exploits the phase of the loop and allows the right-hand side to be unbounded (see Figure \ref{fig:my_label0}). This secant criterion is also necessary for stability when the $a_i$'s are identical. 
\end{remark}

\section{Robustness Analysis of  Stable FLTI Networks}
\label{Sec:5}

{In this section, a framework for robustness analysis of FLTI systems under the influence of external anomalous disturbances is developed.  Our first goal is to calculate the (squared) $\mathcal H_2$ norm for stable commensurate order LTI systems with normal state matrices. Our second goal is to utilize this result to characterize size-dependent fundamental limits on the performance measure of cyclic linear dynamical networks and consensus networks.

\subsection{Robustness Measure}
\label{measure}
In order to find the robustness performance of system $\Sigma$ given by \eqref{eq2}, we utilize the frequency domain definition of (squared) $\mathcal H_2$-norm of system \cite{Doyle89}, {\it i.e.},
	\begin{equation}\label{H2normCalc}
	\rho(\Sigma,\alpha) ~:=~ \frac{1}{2\pi}\Tr\Big[\int_{-\infty}^{+\infty}{G^{\CT}(j\omega)G(j\omega) \: d\omega}\Big]
	\end{equation}
	with transfer matrix   
	\begin{equation}\label{eq:Gs}
	G(s)~=~ C \Big(s^\alpha I_{n}-A \Big)^{-1} B.
	\end{equation}
\begin{remark}
Using Parseval's theorem, we can evaluate the (squared) $\mathcal H_2$ norm in \eqref{H2normCalc} by
computing the time-domain $L_2$-norm of impulse response function
\[ H(t) ~:=~ \mathcal L^{-1}\{G(s)\},\]
where $H(t)$ can be obtained based on the generalized exponential matrix using Mittag-Leffler
function (cf. Theorem 4 in \cite{SiamiFrac2013}). Therefore, the time-domain and frequency-domain $L_2$-norms of $H(t)$ and $G(t)$ are equal, {\it i.e.},  
\begin{eqnarray}\label{H2normPars-1}
	\rho(\Sigma,\alpha) &=& \frac{1}{2\pi}\Tr\Big[\int_{-\infty}^{+\infty}{G^{\CT}(j\omega)G(j\omega) \: d\omega}\Big]\\&=&\int_0^\infty \Tr \left (H(t)H^\top(t)\right) ~dt.\label{H2normPars-2}
\end{eqnarray}
The (squared) $\mathcal H_2$ norm of these systems can be interpreted as the impulse response energy. Specifically, consider $m$ experiments where in each of them we only fed an unit impulse function at the $i$-th input of system $\Sigma$, {\it i.e.}~$\xi (t) = e_i \delta(t)$ where $\{e_1, e_2, \cdots, e_m\}$ is an orthonormal basis of $\R^m$ and $\delta(t)$ is the unit impulse function. Then, we denote the corresponding output by $y_i \in \R^p$. According to \eqref{H2normPars-2}, the (squared) $\mathcal H_2$ norm  is the sum of the $L_2$ norms of these $m$ outputs, {\it i.e.}
\[ 	\rho(\Sigma,\alpha) ~=~ \sum_{i=1}^m \int_0^\infty \left |y_i(t)\right |^2 dt.\]
\end{remark}

In the following result shows how the value of \eqref{H2normCalc} depends on the general properties of state matrix $A$.
\begin{theorem}
\label{th:main-updated}
Suppose that FLTI system \eqref{eq2} is stable, $B~=~C=~I_n$, and matrix $A$ is normal, {\it i.e.}, $A^\top A=AA^\top$. The performance measure \eqref{H2normCalc} for $\frac{1}{2} < \alpha < 2$ is given by 
\[ \rho(\Sigma,\alpha) ~=~ \sum_{i=1}^n \frac{ \sin \left ((\frac{\alpha \pi}{2} - \arg(\lambda_i(-A)))(1-\beta)\right)}{\alpha \,\sin (\frac{\pi}{\alpha} )\sin(\frac{\alpha \pi}{2}-\arg(\lambda_i(-A)))}\, |\lambda_i(A)|^{-\beta},\]
where $\beta=2-\alpha^{-1}$. Moreover, for $\alpha = 1$, this reduces to
\[\lim_{\alpha \rightarrow 1} \rho(\Sigma,\alpha) ~=~ \frac{1}{2}\sum_{i=1}^{n} \left(\Real\{\lambda_i(-A)\}\right )^{-1}. \]
\end{theorem}
\begin{proof}
	In order to find the $\mathcal H_2$-norm of system $\Sigma$ with $B=C=I_n$, we utilize the frequency domain definition of $\mathcal H_2$-norm \eqref{H2normCalc}
	with transfer matrix   
	\begin{align}\label{eq:Gs}
	G(s)~=~ \Big(s^\alpha I_{n}-A\Big)^{-1}.
	\end{align}
	
	We consider spectral decomposition of normal matrix $A$, which is,
	\begin{align*}
	A~=~U \Lambda U^{\CT},
	\end{align*}
	where $U$ is a unitary matrix ({\it i.e.}, $U^{\CT}U=UU^{\CT}=I_n$) of eigenvectors and $\Lambda=\diag(\lambda_1,\ldots,\lambda_{n})$ is the diagonal matrix of eigenvalues of $A$.
	Hence, we have
	\begin{align}\label{eq:GHG-1}
	&\Tr\big[G^{\CT}(j \omega)G(j \omega)\big]\nonumber\\
	=&\Tr\Bigg[  U \diag \Big[\frac{1}{\overline{(j \omega)^\alpha-\lambda_1} },\dots,\frac{1}{\overline{(j \omega)^\alpha-\lambda_{n}} }\Big]\nonumber\\
	\:\:&\diag \Big[\frac{1}{(j \omega)^\alpha-\lambda_1 },\dots,\frac{1}{(j \omega)^\alpha-\lambda_{n} }\Big]
	U^{\CT}\Bigg]
	\end{align}
	and by substituting (\ref{eq:GHG-1}) in \eqref{H2normCalc}, we obtain
	\begin{eqnarray}
	\hspace{-0.15cm}\rho(\Sigma,\alpha)
	= \frac{1}{2\pi} \sum_{i=1}^{n}{\int_{-\infty}^{+\infty}\hspace{-0.3cm}\frac{ d\omega}{\big(\overline{(j \omega)^\alpha-\lambda_i} \big)\big((j \omega)^\alpha-\lambda_i \big)}}.
	\end{eqnarray}
	Let us assume $\lambda_i= |\lambda_i|{\rm e}^{j \arg(\lambda_i)}$.
	Simplifying the integral above
for $\frac{1}{2} < \alpha < 2$ and using Proposition 3 in \cite{malti2011analytical}, 	we obtain 
\[ \rho(\Sigma,\alpha) ~=~ \sum_{i=1}^n \frac{ \sin \left ((\frac{\alpha \pi}{2} - \arg(-\lambda_i))(1-\beta)\right)}{\alpha \,\sin (\frac{\pi}{\alpha} )\sin(\frac{\alpha \pi}{2}-\arg(-\lambda_i))}\, |\lambda_i|^{-\beta},\]
where $\beta=2-\alpha^{-1}$. Moreover, for $\alpha = 1$, this reduces to $\rho(\Sigma,1) ~=~ \frac{1}{2}\sum_{i=1}^{n} \left(\{\Real\{-\lambda_i\}\right)^{-1}$ which is known result for integer-order systems \cite{SiamiAutomatica}.
\end{proof}

We should note that for $\alpha \in (0, 0.5]$ the FLTI system $\Sigma$ might be stable. However, based on Corollary 4 in \cite{malti2011analytical}, its $\mathcal H_2$ norm is unbounded. Moreover, according to Theorem \ref{th:main-updated}, the $\mathcal H_2$ norm tends to infinity as $\alpha$ approaches the instability limit $\alpha = 2$.

}

{ 

\subsection{Linear Networks With Cyclic Interconnection Topology}\label{sec:IV}

As discussed in Section \ref{sec:III}, the class of cyclic networks has been studied in the context of systems biology \cite{tyson1978dynamics}. In order to obtain the (squared) $\mathcal H_2$ norm of the commensurate order cyclic network (\ref{network}) we need to use the result in Section \ref{sec:III} to guarantee stability commensurate order  cyclic networks.

\begin{corollary}
For the cyclic linear dynamical network \eqref{eq66} with $\alpha \in (0,2)$, state matrix 
\begin{eqnarray}
	A = \left[
	\begin{array}{cccccc}
	 -a&0&\cdots&0&-c\\
	  c&-a&\cdots&0&0\\
	 \vdots &&\ddots &&\vdots&\\
 	  0&0&\cdots&-a&0\\
  	 0&0&\cdots&c&-a
	\end{array}
	\right],
	\label{eq6}
\end{eqnarray} and output matrix $C = B = I_n$, assume the stability conditions in Theorem \ref{th-main} hold, {\it i.e.}, $\gamma :=	\frac{c}{ a} < \frac{\sin(\frac{\alpha \pi}{2})}{\sin({\frac{\alpha \pi}{2}-\frac{\pi}{n}})}$ or $\alpha \leq\frac{2}{n}$. Then the (squared) $\mathcal H_2$ norm is given by
\[ \rho(\Sigma,\alpha) ~=~ \sum_{k=1}^n \frac{ \sin \left ((\frac{\alpha \pi}{2} - \arg(-\lambda_k))(\alpha^{-1}-1)\right)}{\alpha \,\sin (\frac{\pi}{\alpha} )\sin(\frac{\alpha \pi}{2}-\arg(-\lambda_k))}\, |\lambda_k|^{\alpha^{-1}-2},\]
where $\lambda_k= -a + c \, {\rm e}^{j\frac{2\pi k}{n}}$ for $k=1,2,\cdots,n$. Moreover, for $\alpha = 1$, this reduces to
\begin{eqnarray}
\rho(\Sigma,1) ~=~ \left\{\begin{array}{ccc}
\frac{n\tan \frac{\beta}{2}}{2 {c} \sin {\frac{\beta}{n}} }&~~\textrm{if}~~&a < c 
\\
\frac{n^2}{4 {c} }&~~\textrm{if}~~&a = c \\
\frac{n\tanh \frac{{\beta}}{2}}{2 {c} \sinh {\frac{\beta}{n}} }&~~\textrm{if}~~&a > c
\end{array}
\right.\label{exact_lower}
	\end{eqnarray}   
%
%
where
	\begin{eqnarray}
		\beta := 
\left\{\begin{array}{ccc}
\mathrm{arcos}(\frac{a}{c}) ~n&~~\textrm{if}~~&a \leq c  \\
\mathrm{arcosh}(\frac{a}{c}) ~n& ~~\textrm{if}~~&a > c
\end{array}\right..				
\label{eq77}
	\end{eqnarray}

\end{corollary}

\begin{proof}
    First note that $A$ is a normal matrix because  $AA^\top = A^\top A$.    Using the fact that $A$ is normal, Theorem \ref{th:main-updated} and \cite[Th. 7]{SiamiAutomatica}, we get the desired result.
\end{proof}
}
\subsection{FLTI Consensus Networks } 
\label{sec:158}

	In this subsection, we consider the class of  linear fractional-order dynamical networks that consist of multiple agents with scalar state variables $x_i$ and control inputs $u_i$ whose dynamics evolve in time according to 	
\begin{eqnarray}
\frac{d^{\alpha}}{d t^{\alpha}}{x}_i(t) & = & u_i(t) +\xi_i(t) \label{TI-consensus-algorithm} \\
y_i(t) & = & x_i(t) - \bar{x}(t)  \label{TI-consensus-algorithm-2}
\end{eqnarray}
for all $i=1,\ldots,n$, where $x_i(0)=x_i^*$ is the initial condition and \[\bar{x}(t)=\frac{1}{n}\big(x_1(t)+\ldots+x_n(t)\big)\] 
is the average of all states at time $t$. The impact of the uncertain environment on each agent's dynamics is modeled by the exogenous anomalous disturbance input $\xi_i(t)$. 
By applying the following feedback control law to the agents of this network 
\begin{equation}
u_i(t) ~=~\sum_{j=1}^{n} k_{ij} \big(x_j(t) - x_i(t)\big),\label{feedback-law}
\end{equation}
the resulting closed-loop system will be a first-order linear consensus network. The closed-loop dynamics of the network (\ref{TI-consensus-algorithm})-\eqref{TI-consensus-algorithm-2} with feedback control law \eqref{feedback-law} can be written in the following compact form
\begin{eqnarray}
\Gamma:
    \begin{cases}
	\frac{d^{\alpha}}{d t^{\alpha}} x(t)  =   -L\, x(t)~+~\xi(t)\label{first-order}\\
	y(t)  =  M_n \, x(t), \label{first-order-G}
	\end{cases}
\end{eqnarray}
%
with  initial condition $x(0)=  x^*$, where  $x = [x_1,  \ldots,  x_n]^{\rm T}$ is the state, $y = [y_1,  \ldots,  y_n]^{\rm T}$ is the output, and $\xi = [\xi_1,  \ldots,  \xi_n]^{\rm T}$ is the anomalous disturbance input of the network. 
The state matrix of the network is a graph Laplacian matrix that is defined by $L=[l_{ij}]$, where 
\begin{equation}
\displaystyle l_{ij} := \left\{\begin{array}{ccc}
-k_{ij} & \textrm{if} & i \neq j \\
 &  &  \\
k_{i1}+\ldots+k_{in}& \textrm{if} & i=j
\end{array}\right.
\end{equation}
and the output matrix is a  centering matrix that is defined by
\begin{equation}
M_n~:=~I_{n} - \frac{1}{n}J_n. 
\end{equation}

The underlying coupling graph of the consensus  network \eqref{first-order}-\eqref{first-order-G} is a graph $\G=(\V,\mathcal E, w)$ with node set $\V=\{1,\ldots,n\}$, edge set 
\begin{equation} 
	\EE=\Big\{ \{i,j\}~\big|~\forall~i,j \in \V,~k_{ij} \neq 0\Big\}, \label{edge-set}
\end{equation}
and weight function $w(e)=k_{ij}$ for all $e=\{i,j\} \in \EE$, and $w(e)=0$ if $e \notin \EE$. The Laplacian matrix of graph $\G$ is equal to $L$. 
\begin{assumption}\label{assump-simple}
The coupling graph $\G$ of the consensus network \eqref{first-order}-\eqref{first-order-G} is connected and time-invariant.
Moreover, all feedback gains (weights) satisfy the following properties for all $i,j \in \V$: 

\vspace{0.1cm}
\noindent (a)~non-negativity: $k_{ij} \geq 0$, \\
\noindent (b)~symmetry: $k_{ij}=k_{ji}$,\\
\noindent (c)~simpleness: $k_{ii}= 0$.
\vspace{0.1cm}
\end{assumption}

Property (b) implies that feedback gains are symmetric and (c) means that there is no self-feedback loop in the network.

According to Assumption \ref{assump-simple}, the underlying coupling graph is undirected, connected, and simple. Assumption \ref{assump-simple} implies that only one of the modes of network \eqref{first-order} is marginally stable  with eigenvector $\mathbbm{1}_n$ and all other ones are stable (see Theorem \ref{th-1} and Definition \ref{def:marg}). The marginally stable mode, which corresponds to the only zero Laplacian eigenvalue of $L$, is unobservable from the output \eqref{first-order-G}. The reason is that the output matrix of the network satisfies $M_n \mathbbm{1}_n= 0$.\\ 
%
\begin{corollary}
\label{th:consensus}
Assume that there is no exogenous noise input, {\it i.e.}, $\xi(t) =  0$ for all time, and Assumption \ref{assump-simple} holds, then the states of all agents converge to a consensus state, which for network $\Gamma $ \eqref{first-order}, the consensus state is
\begin{equation}
\lim_{t \rightarrow \infty} x(t) ~=~ \frac{1}{\alpha}\bar x(0) \mathbbm{1}_n~=~\frac{1}{n\alpha}\mathbbm{1}_n\mathbbm{1}_n^{\text T} ~x^*, \label{limit-zero} %
\end{equation}
where {$\alpha \in (0,2)$}.
\end{corollary}
\begin{proof}
This is a direct consequence of Theorems 4 and 5 of \cite{SiamiFrac2013} and \cite{mainardi2000mittag}.
\end{proof}

\vspace{.1cm}
\begin{remark}
For the case of $\alpha =1$, the result of Corollary \ref{th:consensus} recovers the well-known result of ordinary consensus problems \cite{olfati}.
\end{remark}

When the network is fed with a nonzero exogenous noise input, the limit behavior \eqref{limit-zero} is not expected anymore and the state of all agents will be fluctuating around the consensus state without converging to it. We utilize the definition of $\mathcal H_2$-norm of the network as discussed in Subsection \ref{measure} to capture the effect of noise propagation throughout the network and   quantify degrees to which the state of all agents are dispersed from the consensus state. Let us define
	\begin{equation}\label{H2normCalc-1}
	\rho(\Gamma, \alpha) ~=~ \frac{1}{2\pi}\Tr\Big[\int_{-\infty}^{+\infty}{G^{\CT}(j\omega)G(j\omega) \: d\omega}\Big]
	\end{equation}
	with transfer matrix   
	\begin{equation}\label{eq:Gs}
	G(s)~=~ M_{n} \Big(s^\alpha I_{n}+L \Big)^{-1}.
	\end{equation}
Although $G(s)$ is not asymptotically stable, its single marginally stable mode is not observable in the output, which, consequently results in a bounded $\mathcal H_2$-norm for the network.

\begin{theorem}
\label{th:main2}
Suppose that an FLTI consensus network  \eqref{first-order} over graph $\G$ is given. The performance measure \eqref{H2normCalc-1} for $\frac{1}{2} < \alpha < 2$ is given by 
\[ \rho(\Gamma,\alpha) ~=~ \left|\frac{\cot(\frac{\alpha\pi}{2})}{\alpha \sin(\frac{\pi}{\alpha})}\right| \sum_{i=2}^{n} \lambda_i^{-\beta},\]
where $\beta=2-\alpha^{-1}$. Moreover, for $\alpha = 1$, this reduces to
\[ \rho(\Gamma,1) ~=~ \frac{1}{2}\sum_{i=2}^{n} \lambda^{-1}_i. \]
\end{theorem}
\begin{proof}
	In order to find the performance of the network \eqref{first-order}-\eqref{first-order-G}, we utilize the frequency domain definition of $\mathcal H_2$-norm \eqref{H2normCalc-1}
	with transfer matrix   
	\begin{align}\label{eq:Gs}
	G(s)~=~ M_{n} \Big(s^\alpha I_{n}+L\Big)^{-1}.
	\end{align}
	Although $G(s)$ is not asymptotically stable, its single marginally stable mode is not observable in the output, {resulting in} a bounded $\mathcal H_2$-norm for the network.
	We consider spectral decomposition of Laplacian matrix $L$, which is,
	\begin{align*}
	L~=~Q \Lambda Q^{\T},
	\end{align*}
	where $Q=[q_1,q_2, \dots , q_{n}] \in \mathbb{R}^{n\times n}$ is the orthonormal matrix of eigenvectors and $\Lambda=\diag(\lambda_1,\ldots,\lambda_{n})$ is the diagonal matrix of eigenvalues. We recall that $\lambda_1=0$ for the reason that the graph is undirected and it has no self-loops.
	Therefore,
	\begin{align}
	M_{n}~=&~I_{n}-Q \diag [1,0, \dots,0] Q^{\T}\nonumber\\\label{eq:Mn}
	~=&~Q \diag[0,1, \dots,1] Q^{\T},
	\end{align}
	and
	\begin{align}
	L~=~Q \diag[0,\lambda_2, \dots, \lambda_{n}] Q^{\T}. \label{eigen-decom}
	\end{align}
	Also, substituting \eqref{eq:Mn} and \eqref{eigen-decom} into (\ref{eq:Gs}), we obtain
	\begin{equation}
		G(s)~=~Q \diag\big[0,\frac{1}{s^\alpha+\lambda_2 },\dots,\frac{1}{s^\alpha+\lambda_{n}}\big] Q^{\T} .
	\end{equation}
	Hence, we have
	\begin{align}\label{eq:GHG}
	&\Tr\big[G^{\CT}(j \omega)G(j \omega)\big]\nonumber\\
	=&\Tr\Bigg[  Q \diag \Big[0,\frac{1}{\overline{\lambda_2 +(j \omega)^\alpha}},\dots,\frac{1}{\overline{\lambda_{n} +(j \omega)^\alpha} }\Big]\nonumber\\
	\:\:&\diag \Big[0,\frac{1}{(j \omega)^\alpha+\lambda_2 },\dots,\frac{1}{(j \omega)^\alpha+\lambda_{n} }\Big]
	Q^{\T}\Bigg]
	\end{align}
	and by substituting (\ref{eq:GHG}) in \eqref{H2normCalc-1}, we obtain
	\begin{eqnarray}
	\hspace{-0.15cm}\rho(\Gamma,\alpha)
	= \frac{1}{2\pi} \sum_{i=2}^{n}{\int_{-\infty}^{+\infty}\hspace{-0.3cm}\frac{ d\omega}{\big(\overline{(j \omega)^\alpha+\lambda_i} \big)\big(\lambda_i + (j \omega)^\alpha \big)}}.
	\end{eqnarray}
	Simplifying the integral above
for $\frac{1}{2} < \alpha < 2$ and using Corollary 4 in \cite{malti2011analytical}, 	we obtain 
\[ \rho(\Gamma,\alpha) ~=~ \left |\frac{\cot(\frac{\alpha\pi}{2})}{\alpha \sin(\frac{\pi}{\alpha})}\right| \sum_{i=2}^{n} \lambda_i^{-\beta},\]
where $\beta=2-\alpha^{-1}$. Moreover, for $\alpha = 1$, this reduces to $\rho(\Gamma,1) ~=~ \frac{1}{2}\sum_{i=2}^{n} \lambda^{-1}_i$.
\end{proof}

\vspace{.2cm}
\begin{remark}
The result of Theorem \ref{th:main2} shows that the $\mathcal H_2$-norm of an FLTI consensus network is closely related to the spectral zeta function of its underlying graph.\footnote{For a given Laplacian $L$, its corresponding spectral zeta function of order $q \geq 1$ is defined by $\zeta_q (L) := \left( \sum_{i=2}^n \lambda_i^q \right)^{\frac{1}{q}}$.} 
The performance measure \eqref{H2normCalc-1} also relates to the concept of coherence in consensus networks and the expected dispersion of the state of the system in steady state \cite{SiamiTAC2016,Bamieh12,SiamiTAC2018}. In the case that $\alpha =1$, it also has close connections to the total effective resistance of graph \cite{Bamieh12,SiamiTAC2016}.
\end{remark}

\section{Discussion}
In this section, we consider several numerical examples to illustrate our result in Section \ref{main}.

{
\begin{example}
\label{ex-1}
\begin{figure}
    \centering
	\includegraphics[trim = 60.5 5.5 10.5 5.5, clip,width=.7 \textwidth]{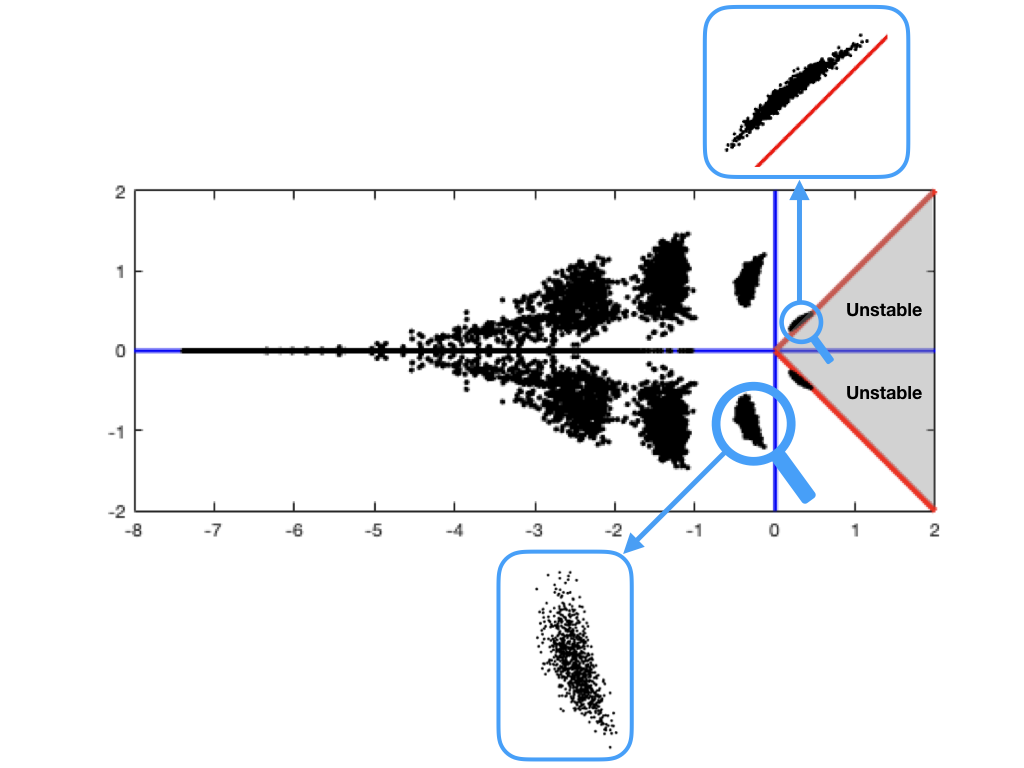}
    \caption{We consider $1000$ randomly generated networks with $10$ nodes, $\gamma = 1.5575$, and $\theta=2$ in Example \ref{ex-1}. Black dots show eigenvalues of $1000$ randomly generated state matrices $A$ with sparsity pattern given by \eqref{eq6} where $\frac{\mathfrak c}{\mathfrak a} = \gamma$. The gray area shows the unstable area, {\it i.e.},  $|\text{arg}(\lambda)|< \frac{\alpha \pi}{2}$, and red line segments show $|\text{arg}(\lambda)| = \frac{\alpha \pi}{2}$. {The zoomed-in areas for the clusters are presented by the blue rounded rectangles (magnified three times).}}
    \label{fig:my_label}
\end{figure}
\begin{figure}
    \centering
	\includegraphics[trim = 90.5 5.5 90.5 5.5, clip,width=.7 \textwidth]{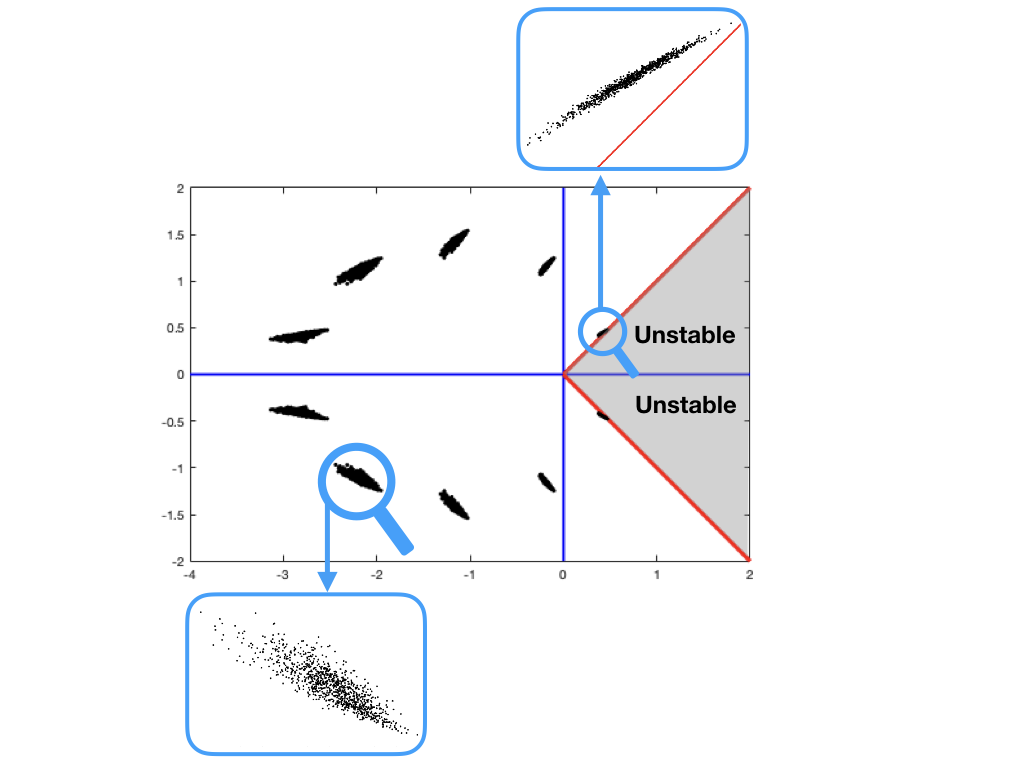}
    \caption{We consider $1000$ randomly generated networks with $10$ nodes, $\alpha=0.5$, $\gamma=\frac{\mathfrak c}{\mathfrak a} =1.5575$, and $\theta=2$ in Example \ref{ex-1}. Black dots show eigenvalues of the state matrices in the complex plane. The gray area shows the unstable area, {\it i.e.},  $|\text{arg}(\lambda)|< \frac{\alpha \pi}{2}$, and red line segments demonstrate $|\text{arg}(\lambda)| = \frac{\alpha \pi}{2}$. {The zoomed-in areas for the clusters are presented by the blue rounded rectangles (magnified four times).}}
    \label{fig:my_label11}
\end{figure}
In this example, for each parameter $\theta \in \{2,1\}$, we consider $1000$ randomly generated networks in the form of given by  \eqref{network} with $n=10$ and $\alpha =0.5$. To do so, we randomly generate $1000$ state matrices with sparsity pattern given by dynamics \eqref{eq6}. Specifically, we randomly generate random vectors of values, $[\log a_1, \cdots, \log a_n]^\top$, each with a fixed sum $0$, and subject to the following restrictions
\[-\theta ~\leq~ \log a_i ~\leq~ \theta,\]
where $\theta \in \{1,2\}$.
Thus, for $a_i$'s we have
\begin{eqnarray}
\prod_{i=1}^{n}a_i ~=~ \exp(0) ~=~ 1,
\label{eq:402}
\end{eqnarray} 
and $\exp(-\theta) \leq a_i \leq \exp(\theta)$.
Similarly, we randomly generate $n$-element vectors of values, $[\log c_1, \cdots, \log c_n]^\top$, each with a fixed sum $n \times \log \gamma $, and subject to the restriction
\[-\theta \leq \log c_i \leq \theta,\]
where $\theta \in \{1,2\}$. 
Then, it follows that
\begin{equation}
    \prod_{i=1}^{n}c_i ~=~ \exp \left(\sum_{i=1}^n \log c_i\right)~=~\exp(n \log \gamma) ~=~ \gamma^n, 
    \label{eq:410}
\end{equation} 
and $\exp(-\theta) \leq c_i \leq \exp(\theta)$. We then assume that $\gamma = 1.5575$.
Using \eqref{eq:402}, \eqref{eq:410}, and \eqref{eq4}, we have
\begin{equation*}
    \gamma ~=~ 1.5575  ~<~ \frac{\sin(\frac{\pi}{4})}{\sin(\frac{\pi}{4}-\frac{\pi}{10})} ~=~ 1.55753651.
\end{equation*}
Therefore, the secant condition \eqref{eq4} holds, and all the systems in this example are stable (cf. Theorem \ref{th-1}). Figures \ref{fig:my_label} and \ref{fig:my_label11} show poles of these $1000$ randomly generated networks with $10$ nodes in the complex plane ({\it i.e.}, eigenvalues of state matrix $A$) for $\theta =2$ and $\theta=1$, respectively. The poles are shown by black dots. These clouds of poles ($10\times1000$ poles) are clustered in several dense regions. In Figures \ref{fig:my_label} and \ref{fig:my_label11}, the zoomed areas for some of the clusters are {presented}. As we expected, based on Theorem \ref{th-main}, all poles are located in stable area.
\end{example}

\vspace{.1cm}
{\begin{remark}
According to Figures \ref{fig:my_label} and \ref{fig:my_label11}, as the interval size for parameters $a_i$'s and $c_i$'s decreases (having less dispersion for $a_i$'s and $c_i$'s), the cloud of poles becomes more regular/symmetrical, and the clusters of poles emerge.
\end{remark}}
\vspace{.1cm}

Finally, we consider a similar setup as Example \ref{ex-1} but with a different value for parameter $\gamma$ to generate unstable systems.

\begin{example}
\label{ex-2}
\begin{figure}
    \centering
	\includegraphics[trim = 90.5 105.5 40.5 125.5, clip,width=.65 \textwidth]{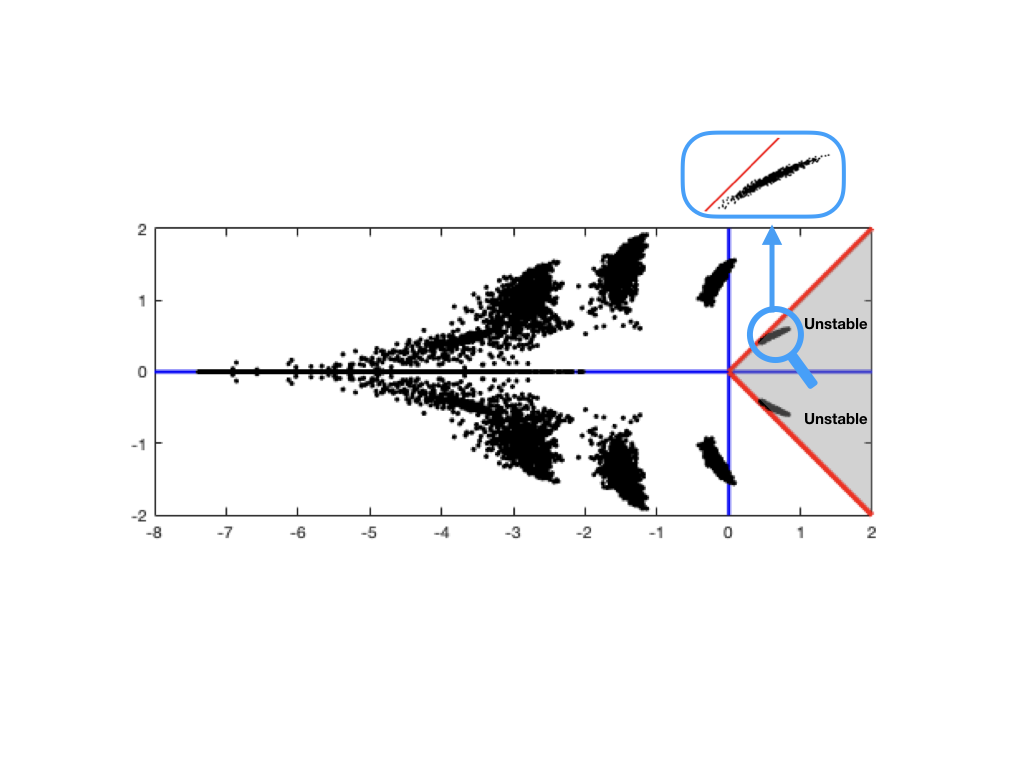}
    \caption{We consider $1000$ randomly generated networks \eqref{network} with $10$ nodes, $\alpha=0.5$ and $\gamma=\frac{\mathfrak c}{\mathfrak a} =2$ in Example \ref{ex-2}. Black dots show eigenvalues of the state matrices in the complex plane. The gray area shows the unstable area, {\it i.e.},  $|\text{arg}(\lambda)|< \frac{\alpha \pi}{2}$, and red line segments demonstrate $|\text{arg}(\lambda)| = \frac{\alpha \pi}{2}$. {The zoomed-in area for the cluster is presented by the blue rounded rectangle (magnified three times). }}
    \label{fig:my_label3}
\end{figure}
In this example, we use the same setup as Example \ref{ex-1} with $\theta =1$ and $\gamma =   2$.
Using \eqref{eq:402}, \eqref{eq:410}, and \eqref{eq4}, we get
\begin{equation}
    \gamma ~=~ 2 ~>~ \frac{\sin(\frac{\pi}{4})}{\sin(\frac{3\pi}{20})}.
\end{equation}
This means that the generalized secant condition \eqref{eq4} in Theorem \ref{th-main} does not hold. As shown in Figure \ref{fig:my_label3}, all the randomly generated systems are unstable. 

Figure \ref{fig:my_label3} shows poles of network \eqref{network} with $n=10$ and $\alpha =0.5$ in the complex plane ({\it i.e.}, eigenvalues of state matrix $A$). The poles are shown by black dots. These clouds of poles ($10\times1000$ poles) are clustered in several dense regions. In Figure \ref{fig:my_label3}, the zoomed-in areas for some of the clusters are shown. As we expected, based on the result of Theorem \ref{th-main}, all poles are not located in the stable area (see Theorem \ref{th-1}).
\end{example}
}

\section{Conclusions and Future Research Directions}

We have presented a stability  criterion for a class of interconnected fractional-order systems, which encompasses the secant criterion for cyclic LTI systems. This result presents a sufficient condition for stability of networks of cyclic interconnection of fractional-order systems when the digraph describing the network conforms to a single circuit. 
{Furthermore, we have investigated the robustness of fractional-order linear networks with normal state matrices. The robustness performance of an FLTI network is quantified using $\mathcal H_2$-norm of the dynamical system, and a closed-form expression based on the spectral zeta function of the underlying graph is presented. We then discuss two important sub-classes of these networks: cyclic networks and consensus networks.}
Finally, the theoretical results are confirmed via some numerical illustrations. 

{For our future research direction, we envision that tools from passivity gains and the secant condition for stability \cite{sontag2006} could be extended to the network of fractional-order systems, and we hope to utilize these results to devise a more accurate model for soft robots. An accurate analytic model of a soft robot will generate fast and reliable data, and therefore will allow us to obtain more practical control and estimation algorithms.}

\begin{spacing}{1.5}
\bibliography{main_Milad}
\end{spacing}
\end{document}